\documentclass{amsart}
\usepackage{amsmath, amsthm, amsfonts}
\usepackage{verbatim}

\numberwithin{equation}{section}
\newtheorem{lemma}{Lemma}[section]

\newtheorem{theorem}{Theorem}[section]

\newcommand{\lgn}{\frac{\log{\log{n}}}{n}}
\newcommand{\elg}{\frac{1}{\log{n}}}
\newcommand{\Exp}{\textup{E}}

\newcommand{\bfl}{\begin{flalign*}}
\newcommand{\efl}{\end{flalign*}}
\newcommand{\cont}{C\left([-\pi,\pi]\right)}
\newcommand{\diff}{C^1\left([-\pi,\pi]\right)}

\begin{document}

\title[Level crossings of a random algebraic polynomial]{The K-level crossings of a random algebraic polynomial with dependent coefficients}
\author[J. Matayoshi]{Jeffrey Matayoshi}
\address{Department of Mathematics\\
340 Rowland Hall\\
University of California, Irvine\\
Irvine, CA 92697-3875}
\email{jsmatayoshi@gmail.com}
\thanks{This research was partially supported by NSF grant DMS-0706198}

\begin{abstract}
For a random polynomial with standard normal coefficients, two cases of the $K$-level crossings have been considered by Farahmand.  When the coefficients are independent, Farahmand was able to derive an asymptotic value for the expected number of level crossings, even if $K$ is allowed to grow to infinity.  Alternatively, it was shown that when the coefficients have a constant covariance, the expected number of level crossings is reduced by half.  In this paper we are interested in studying the behavior for dependent standard normal coefficients where the covariance is decaying and no longer constant.  Using techniques similar to those of Farahmand, we will be able to show that for a wide range of covariance functions behavior similar to the independent case can be expected.
\end{abstract}

\maketitle

\section{Introduction}

For the random polynomial given by 
\begin{equation} \label{poly}
P_n(x)=\sum_{k=0}^{n}{X_{k}x^k},
\end{equation}
consider the problem of computing the expected number of real zeros for the equation $P_n(x)=K$, where $K$ is a given constant.  These are known as the $K$-level crossings of $P_n(x)$.  For standard normal coefficients, Farahmand considered two separate cases in \cite{Farahmand86} and \cite{Farahmand862}.  The first assumes the coefficients are independent.  Here, Farahmand derived an asymptotic value for the expected number of level crossings, for both $K$ bounded and $K$ growing with $n$.  The second case deals with dependent coefficients with a constant covariance $\rho$, where $\rho \in (0,1)$.  What Farahmand showed here was that the constant covariance causes the expected number of level crossings to be reduced by half.  With that in mind, the goal of this paper is to further study the case of dependent coefficients.  We are interested in the behavior of the crossings when there is some decay of the covariance between the coefficients.  

The setup for this problem will be as follows.  Let $X_0, X_1, \ldots$ be a stationary sequence of normal random variables, where the covariance function is given by
\[
\Gamma(k) = \Exp\left[X_0X_k\right], \qquad \Gamma(0) = 1.
\]
Similar to our investigation in \cite{Matayoshi}, we will express $\Gamma(k)$ using the spectral density.  That is,
\begin{equation} \label{cov}
\Gamma(k)=\int_{-\pi}^{\pi}e^{-ik\phi}f(\phi)d\phi,
\end{equation}
where $f(\phi)$ is the spectral density of the covariance function (in addition to the discussion in \cite{Matayoshi}, see \cite{Brei} and \cite{CramLead} for further references).  By imposing certain conditions on the spectral density, for the random polynomial $P_n(x)$ given by \eqref{poly}, we will be able to study the level crossings for a wide range of covariance functions.

Our work will cover two different assumptions on $K$, similar to those considered by Farahmand.  As long as the spectral density has nice enough properties, similar behavior to the independent case can be expected.  Assuming $K$ is bounded, if we require that the spectral density is positive and in $\cont$, we will be able to show that the expected number of level crossings will behave asymptotically like $\frac{2}{\pi}\log{n}$ as $n\rightarrow \infty$.  On the other hand, if $K$ is allowed to grow along with $n$, such that $K=o\left(\sqrt{\frac{n}{\log{\log{n}}}}\right)$, and if the spectral density is positive and in $\diff$, the expected number of crossings in the interval $(-1,1)$ is reduced.  These results will be proved using the techniques developed by Farahmand in \cite{Farahmand86} and \cite{Farahmand862}, as well as the spectral density of the covariance function.  We will also make use of several results from \cite{Matayoshi}, which in turn draws heavily from the work of Sambandham in \cite{Sam77}.  Letting $N_K(\alpha,\beta)$ be the number of $K$-level crossings of $P_n(x)$ in the interval $(\alpha,\beta)$, the main theorem is formulated as follows.
\begin{theorem} \label{Kcrossings}
Assume that the spectral density exists and is strictly positive.
\renewcommand{\theenumi}{\roman{enumi}}
\renewcommand{\labelenumi}{(\theenumi)}
\begin{enumerate}
\item For $K$ bounded and $f(\phi)\in C([-\pi,\pi])$ we have
\[
\Exp\left[N_K\left(-1,1\right)\right] = \Exp\left[N_K\left(-\infty,-1\right)+N_K\left(1,\infty\right)\right] \sim \frac{1}{\pi}\log{n}.
\]
\item For $K=o\left(\sqrt{\frac{n}{\log{\log{n}}}}\right)$ and $f(\phi) \in C^1([-\pi,\pi])$ we have
\begin{flalign*}
\Exp\left[N_K\left(-1,1\right)\right] &= \frac{1}{\pi}\log{\frac{n}{K^2}} + O\left(\log{\log{n}}\right),\\
\Exp\left[N_K\left(-\infty,-1\right)+N_K\left(1,\infty\right)\right] &= \frac{1}{\pi}\log{n} + O\left(\log{\log{n}}\right).
\end{flalign*}
\end{enumerate}
\end{theorem}

To begin with, using the Kac-Rice formula derived in \cite{Farahmand}, we have
\begin{equation} \label{fullKac}
\begin{split}
\Exp\left[N_K\left(\alpha,\beta\right)\right] &= \frac{1}{\pi}\int_{\alpha}^{\beta}\frac{\sqrt{AC-B^2}}{A}\exp\left(-\frac{K^2C}{2\left(AC - B^2\right)}\right)dx\\
	& \quad + \frac{1}{\pi}\int_{\alpha}^{\beta}\frac{\sqrt{2}|BK|}{A^{3/2}}\exp\left(-\frac{K^2}{2A}\right)\mbox{erf}\left(\frac{|-BK|}{\sqrt{2A\left(AC - B^2\right)}}\right)dx\\
	&= \int_{\alpha}^{\beta}F_1dx + \int_{\alpha}^{\beta}F_2dx,\\
\end{split}
\end{equation}
where
\begin{flalign*}
A(x) &= \textup{E}[P_n^2(x)] = \sum_{k=0}^{n}{\sum_{j=0}^{n}{\Gamma(k-j)x^{k+j}}},\\
B(x) &= \textup{E}[P_n(x)P_n'(x)] = \sum_{k=0}^{n}{\sum_{j=0}^{n}{\Gamma(k-j)kx^{k+j-1}}},\\
C(x) &= \textup{E}[(P_n'(x))^2] = \sum_{k=0}^{n}{\sum_{j=0}^{n}{\Gamma(k-j)kjx^{k+j-2}}}.\\
\end{flalign*}
Applying \eqref{cov} gives us
\begin{flalign*}
A &= \int_{-\pi}^{\pi}{\sum_{k=0}^{n}{\sum_{j=0}^{n}{e^{-i(k-j)\phi}x^{k+j}f(\phi)d\phi}}},\\
B &= \int_{-\pi}^{\pi}{\sum_{k=0}^{n}{\sum_{j=0}^{n}{e^{-i(k-j)\phi}kx^{k+j-1}f(\phi)d\phi}}},\\
C &= \int_{-\pi}^{\pi}{\sum_{k=0}^{n}{\sum_{j=0}^{n}{e^{-i(k-j)\phi}kjx^{k+j-2}f(\phi)d\phi}}}.\\
\end{flalign*}
From (2.3), (2.4), and (2.5) in \cite{Matayoshi}, we have
\begin{flalign*}
A &= \int_{-\pi}^{\pi}\frac{1-x^{n+1}e^{-i(n+1)\phi}}{1-xe^{-i\phi}} \cdot \frac{1-x^{n+1}e^{i(n+1)\phi}}{1-xe^{i\phi}}f(\phi)d\phi,\\
B &= \int_{-\pi}^{\pi}\left(\frac{1-x^{n+1}e^{-i(n+1)\phi}}{1-xe^{-i\phi}}\right)\\
  &  \quad \cdot \left(\frac{-(n+1)x^ne^{i(n+1)\phi}(1-xe^{i\phi})-(1-x^{n+1}e^{i(n+1)\phi})(-e^{i\phi})}{(1-xe^{i\phi})^2}\right)f(\phi)d\phi,\\
\end{flalign*}
and
\begin{flalign*}
C &= \int_{-\pi}^{\pi}\left(\frac{-(n+1)x^ne^{-i(n+1)\phi}(1-xe^{-i\phi})-(1-x^{n+1}e^{-i(n+1)\phi})(-e^{-i\phi})}{(1-xe^{-i\phi})^2}\right)\\
  &  \quad \cdot \left(\frac{-(n+1)x^ne^{i(n+1)\phi}(1-xe^{i\phi})-(1-x^{n+1}e^{i(n+1)\phi})(-e^{i\phi})}{(1-xe^{i\phi})^2}\right)f(\phi)d\phi.\\
\end{flalign*}

\section{Expected Number of Level Crossings on $(-1,1)$}

To prove Theorem \ref{Kcrossings} we will start as Farahmand did in \cite{Farahmand86} and \cite{Farahmand862}.  That is, our first step will be to show that the contribution from the integral of $F_2$ on $(-1,1)$ is negligible.  
\begin{lemma} \label{F2bound}
For $f(\phi)$ continuous and positive we have
\[
\int_{-1}^{1}F_2dx = o(\log{\log{n}}).
\]
\end{lemma}
\begin{proof}
Since $f(\phi)$ is a continuous, positive function, we can find constants $c_1,c_2>0$ such that $\frac{c_1}{2\pi}>f(\phi)>\frac{c_2}{2\pi}$ for any $\phi \in[-\pi,\pi]$.  Now, for the interval $(-1+\lgn,1-\lgn)$ we have
\[
A \sim \int_{-\pi}^{\pi}\frac{1}{(1-xe^{-i\phi})(1-xe^{i\phi})}f(\phi)d\phi,
\]
from which we can then derive the lower bound
\begin{equation} \label{lowA}
A \geq \frac{c_2}{2\pi}\int_{-\pi}^{\pi}\frac{1}{(1-xe^{-i\phi})(1-xe^{i\phi})}d\phi = \frac{c_2}{1-x^2} \geq \frac{c_2\left(1-x^{2n+2}\right)}{1-x^2}.
\end{equation}
Using the fact that $f\equiv\frac{1}{2\pi}$ in the independent case, we can derive an upper bound as well, where
\begin{equation} \label{upA}
\begin{split}
A &\leq \frac{c_1}{2\pi}\int_{-\pi}^{\pi}\frac{\left(1 - x^{n+1}e^{-i(n+1)\phi}\right)\left(1 - x^{n+1}e^{i(n+1)\phi}\right)}{(1 - xe^{-i\phi})(1 - xe^{i\phi})}d\phi\\
	&=	  c_1\frac{1-x^{2n+2}}{1-x^2} \leq \frac{c_1}{1-x^2}.
\end{split}
\end{equation}
Notice that this upper bound holds on the entire interval $(-1,1)$.  Next, from equations (3.5) and (3.7) in \cite{Matayoshi} we know that
\[
|B| \sim \int_{-\pi}^{\pi}\left|\frac{e^{i\phi}}{(1 - xe^{-i\phi})(1 - xe^{i\phi})^2}\right|f(\phi)d\phi,
\]
which implies
\begin{equation*}
|B|	\leq \frac{1}{1-|x|}\int_{-\pi}^{\pi}\frac{1}{(1 - xe^{-i\phi})(1 - xe^{i\phi})}f(\phi)d\phi \sim \frac{1}{1-|x|}A.
\end{equation*}
It follows that
\begin{equation*}
\frac{|B|}{A^{3/2}} \leq \frac{1}{1-|x|}\left(\frac{1-x^2}{c_2}\right)^{1/2} \leq \sqrt{\frac{2}{c_2}}\frac{1}{(1-|x|)^{1/2}},											\end{equation*}
and
\begin{equation*}
\exp{\left(\frac{-K^2}{2A}\right)} \leq \frac{1}{1 + \frac{K^2\left(1-x^2\right)}{2c_1}} \leq \frac{1}{1 + \frac{K^2\left(1-|x|\right)}{2c_1}}.
\end{equation*}
Since $\textup{erf}(x)\leq1$, we then have
\begin{equation} \label{F21}
\begin{split}
\int_{-1+\lgn}^{1-\lgn}F_2 dx &\leq \sqrt{\frac{2}{c_2}}\int_{-1+\lgn}^{1-\lgn}\frac{|K|(1-|x|)^{-1/2}}{1 + \frac{K^2(1-|x|)}{2c_1}}dx\\
	&= 2\sqrt{\frac{2}{c_2}}\int_{0}^{1-\lgn}\frac{|K|(1-x)^{-1/2}}{1 + \frac{K^2(1-x)}{2c_1}}dx\\
	&= -2\sqrt{2c_1}\arctan{\left(\frac{K\sqrt{1-x}}{\sqrt{2c_1}}\right)}\bigg|_0^{1-\lgn}\\
	&= O(1).
\end{split}
\end{equation}
Next, for $x\in(-1,-1+\lgn)\cup(1-\lgn,1)$,
\begin{equation*}
|B| \leq \frac{n}{|x|}\sum_{k=0}^{n}{\sum_{j=0}^{n}{\Gamma(k-j)|x|^{k+j}}} \leq \frac{nc_1}{|x|}\sum_{k=0}^{n}x^{2k},
\end{equation*}
by \eqref{upA}.  Also,
\begin{equation*}
A \geq \frac{c_2}{2\pi}\int_{-\pi}^{\pi}\frac{\left(1 - x^{n+1}e^{-i(n+1)\phi}\right)\left(1 - x^{n+1}e^{i(n+1)\phi}\right)}{(1 - xe^{-i\phi})(1 - xe^{i\phi})}d\phi = c_2\sum_{k=0}^{n}x^{2k},
\end{equation*}
from which it then follows that
\begin{flalign*}
\frac{|B|}{A^{3/2}} &\leq    nc\left(\sum_{k=0}^{n}x^{2k}\right)^{-1/2}\\
										&\leq    nc\left(\sum_{k=0}^{n}\left(1-\lgn\right)^{2k}\right)^{-1/2}\\
										&\sim    c\left(n\log{\log{n}}\right)^{1/2}.
\end{flalign*}
Thus,
\begin{flalign*}
\frac{\sqrt{2}}{\pi}\int_{1-\lgn}^1F_2 &\leq \frac{\sqrt{2}}{\pi}\int_{1-\lgn}^{1}\frac{|KB|}{A^{3/2}}\\
		&\leq {\pi}\int_{1-\lgn}^{1}c|K|\left(n\log{\log{n}}\right)^{1/2}\\
		&= o\left(\log{\log{n}}\right).
\end{flalign*}
Similarly,
\[
\frac{\sqrt{2}}{\pi}\int_{-1}^{-1+\lgn}F_2 = o(\log{\log{n}}),
\]
which proves the claim.
\end{proof}
We will next show that the expected number of crossings on the intervals $(0,1-\elg)$, $(1-\frac{\log{\log{n}}}{n},1)$, $(-1+\elg,0)$ and $(-1,-1+\frac{\log{\log{n}}}{n})$ is negligible.
\begin{lemma} \label{Kbound}
Assume $f(\phi)$ is continuous and positive.  For the intervals $(-1,-1+\frac{\log{\log{n}}}{n})$, $(-1+\elg,0)$, $(0,1-\elg)$, and $(1-\frac{\log{\log{n}}}{n},1)$, the expected number of crossings is $O(\log{\log{n}})$.
\end{lemma}
\begin{proof}
To start, we note that since the quantity $\frac{K^2C}{AC-B^2}$ is never negative, the inequality
\[
\exp\left(-\frac{K^2C}{2\left(AC - B^2\right)}\right) \leq 1
\]
holds in general.  It follows that
\begin{equation} \label{F1bound}
\int_{\alpha}^{\beta} F_1 dx \leq \frac{1}{\pi}\int_{\alpha}^{\beta}\frac{\sqrt{AC-B^2}}{A}dx.
\end{equation}
Applying Lemma \ref{F2bound} from above, along with Lemma 2.1 from \cite{Matayoshi}, we then have
\[
\Exp\left[N\left(-1+\elg,1-\elg\right)\right] = O\left(\log{\log{n}}\right),
\]
and
\[
\Exp\left[N\left(-1,-1+\lgn\right)\right] = \Exp\left[N\left(1-\lgn,1\right)\right] = O\left(\log{\log{n}}\right).
\]
\end{proof}
The last lemma of this section will be concerned with computing the integral of $F_1$ on the intervals $(-1+\lgn,-1+\elg)$ and $(1-\elg, 1 - \lgn)$.
\begin{lemma} \label{-1to1}
The integral of $F_1$ on the intervals $(-1+\frac{\log{\log{n}}}{n}, -1+\elg)$ and $(1-\elg, 1-\frac{\log{\log{n}}}{n})$ is given by the following:
\renewcommand{\theenumi}{\roman{enumi}}
\renewcommand{\labelenumi}{(\theenumi)}
\begin{enumerate}
\item For $K$ bounded and $f\in C\left([-\pi,\pi]\right)$,
\begin{equation*}
\frac{1}{\pi}\int_{-1+\lgn}^{-1+\elg}F_1 = \frac{1}{\pi}\int_{1-\elg}^{1-\lgn}F_1 \sim \frac{1}{2\pi}\log{n}.
\end{equation*}
\item For $K=o\left(\sqrt{\frac{n}{\log{\log{n}}}}\right)$ and $f\in C^1\left([-\pi,\pi]\right)$,
\begin{equation*}
\frac{1}{\pi}\int_{-1+\lgn}^{-1+\elg}F_1 = \frac{1}{\pi}\int_{1-\elg}^{1-\lgn}F_1 = \frac{1}{2\pi}\log{\left(\frac{n}{K^2}\right)} + O\left(\log\log{{n}}\right).
\end{equation*}
\end{enumerate}
\end{lemma}
\begin{proof}
We will follow a similar procedure to that used by Farahmand in \cite{Farahmand86} and \cite{Farahmand862}.  That is, an asymptotic value for the integral of $F_1$ will be obtained by deriving upper and lower bounds for the integral, whereupon the true asymptotic value will then lie between these.  Let $g(y)=y\frac{\log{n}}{\log{\log{n}}}$.  Starting with $x=1-y\in(1-\elg,1-\lgn)$, from (3.2), (3.5), and (3.9) in \cite{Matayoshi} we have the equations
\begin{equation} \label{ABC+cont}
\begin{split}
A &\sim \frac{2f(0)}{y}\arctan{\left(\frac{g(y)}{y}\right)},\\
B &\sim \frac{f(0)}{y^2}\arctan{\left(\frac{g(y)}{y}\right)},\\
C &\sim \frac{f(0)}{y^3}\arctan{\left(\frac{g(y)}{y}\right)},
\end{split}
\end{equation}
for $f(\phi)\in \cont$, and
\begin{equation} \label{ABC+}
\begin{split}
A &= \frac{2f(0)}{y}\arctan{\left(\frac{g(y)}{y}\right)} + O\left(\frac{1}{g(y)}\right),\\
B &= \frac{f(0)}{y^2}\arctan{\left(\frac{g(y)}{y}\right)} + O\left(\frac{1}{yg(y)}\right),\\
C &= \frac{f(0)}{y^3}\arctan{\left(\frac{g(y)}{y}\right)} + O\left(\frac{1}{y^2g(y)}\right),
\end{split}
\end{equation}
for $f(\phi)\in \diff$.  We then have the expressions
\begin{equation} \label{ac-b2cont}
\begin{split}
AC - B^2 &\sim \frac{f^2(0)}{y^4}\left[\arctan{\left(\frac{g(y)}{y}\right)}\right]^2,\\
\frac{\sqrt{AC - B^2}}{A} &\sim \frac{1}{2y},
\end{split}
\end{equation}
for $f(\phi)\in \cont$, and
\begin{equation} \label{ac-b2diff}
\begin{split}
AC - B^2 &= \frac{f^2(0)}{y^4}\left[\arctan{\left(\frac{g(y)}{y}\right)}\right]^2 + O\left(\frac{1}{y^3g(y)}\right),\\
\frac{\sqrt{AC - B^2}}{A} &= \frac{1}{2y} + O\left(\frac{1}{g(y)}\right),
\end{split}
\end{equation}
for $f(\phi)\in \diff$.  

We will first handle the simpler case when $f(\phi) \in \cont$ and $K$ is bounded.  Applying \eqref{ABC+cont} and \eqref{ac-b2cont} to \eqref{fullKac} gives
\begin{equation}	\label{F1cont}
\begin{split}
\frac{1}{\pi}\int_{1-\elg}^{1-\lgn}F_1 &= \frac{1}{\pi}\int_{\lgn}^{\elg}\frac{1}{2y}\exp{\left(\frac{-K^2y}{2f(0)\arctan{\left(\frac{g(y)}{y}\right)}}\right)}dy\\
&\sim \frac{1}{\pi}\int_{\lgn}^{\elg}\frac{1}{2y}\left(1 - \frac{K^2y}{2f(0)\arctan{\left(\frac{g(y)}{y}\right)}}\right)dy\\
&\sim \frac{1}{2\pi}\log{n}.
\end{split}
\end{equation}

Next, let $f(\phi) \in \diff$ and $K=o\left(\sqrt{\frac{n}{\log{\log{n}}}}\right)$.  Applying \eqref{ABC+} and \eqref{ac-b2diff} yields
\begin{flalign*}
\frac{CK^2}{2(AC - B^2)} &= \frac{K^2}{2}\left[\frac{f(0)}{y^3}\arctan{\left(\frac{g(y)}{y}\right)} + O\left(\frac{1}{y^2g(y)}\right)\right]\\
	&  \quad \cdot \left[\frac{f^2(0)}{y^4}\left[\arctan{\left(\frac{g(y)}{y}\right)}\right]^2 + O\left(\frac{1}{y^3g(y)}\right)\right]^{-1}\\
	&= \frac{K^2y}{2f(0)\arctan{\left(\frac{g(y)}{y}\right)}} + O\left(\frac{K^2y^2}{g(y)}\right).
\end{flalign*}
Now, we can choose positive constants $a_1$ and $a_2$ such that for large enough n,
\begin{flalign*}
\frac{a_1K^2y}{2f(0)\arctan{\left(\frac{g(y)}{y}\right)}} &\leq \frac{K^2y}{2f(0)\arctan{\left(\frac{g(y)}{y}\right)}} + O\left(\frac{K^2y^2}{g(y)}\right)\\
	& \leq \frac{a_2K^2y}{2f(0)\arctan{\left(\frac{g(y)}{y}\right)}},
\end{flalign*}
which then yields
\begin{multline} \label{F1bounds}
\left[\frac{1}{2y} + O\left( \frac{1}{g(y)}\right)\right]\exp{\left(\frac{-a_2K^2y}{2f(0)\arctan{\left(\frac{g(y)}{y}\right)}}\right)}\\
\leq F_1 \leq \left[\frac{1}{2y} + O\left( \frac{1}{g(y)}\right)\right]\exp{\left(\frac{-a_1K^2y}{2f(0)\arctan{\left(\frac{g(y)}{y}\right)}}\right)}.\\
\end{multline}
For $i=1,2$ we have
\begin{multline} \label{eqn}
\left[\frac{1}{2y} + O\left( \frac{1}{g(y)}\right)\right]\exp{\left(\frac{-a_iK^2y}{2f(0)\arctan{\left(\frac{g(y)}{y}\right)}}\right)} \\ =\frac{1}{2y}\exp{\left(\frac{-a_iK^2y}{2f(0)\arctan{\left(\frac{g(y)}{y}\right)}}\right)} + O\left( \frac{1}{g(y)}\right).
\end{multline}
Thus, using an argument similar to the one on page 706 in \cite{Farahmand86},
\begin{equation} \label{int}
\begin{split}
&\frac{1}{\pi}\int_{\lgn}^{\elg}\Bigg[\frac{1}{2y}\exp{\left(\frac{-a_iK^2y}{2f(0)\arctan{\left(\frac{g(y)}{y}\right)}}\right)} + O\left(\frac{1}{g(y)}\right)\Bigg]dy\\
 			&= \frac{1}{\pi}\int_{\lgn}^{\elg}\frac{1}{2y}\exp{\left(-cK^2y\right)}dy + O\left(\log{\log{n}}\right)\\
 			&  \quad \left(\mbox{where $c=a_i\left[2f(0)\arctan{\left(\frac{g(y)}{y}\right)}\right]^{-1}$}\right)\\
 			&= \frac{1}{2\pi}\left[\log{\left(cK^2\elg\right)} - \log{\left(cK^2\lgn\right)}\right]\\
 			&  \quad + \frac{1}{2\pi}\int_0^{cK^2\lgn}\frac{1-e^{-t}}{t}dt - \frac{1}{2\pi}\int_0^{cK^2\elg}\frac{1-e^{-t}}{t}dt\\
 			&= \frac{1}{2\pi}\log{n} + \frac{1}{2\pi}\int_0^{cK^2\lgn}\frac{1-e^{-t}}{t}dt
 			 - \frac{1}{2\pi}\int_0^{cK^2\elg}\frac{1-e^{-t}}{t}dt + O\left(\log{\log{n}}\right).
\end{split}
\end{equation}
Since we are assuming that $K^2\lgn\rightarrow 0$ as $n\rightarrow \infty$, the first integral is $o(1)$.  For the second we have, by again using an argument drawn from page 706 in \cite{Farahmand86},
\begin{equation} \label{int2}
\begin{split}
		&= - \frac{1}{2\pi}\int_1^{cK^2\elg} \frac{1-e^{-t}}{t}dt - \frac{1}{2\pi}\int_0^{1}\frac{1-e^{-t}}{t}dt \\
		&= -\frac{1}{2\pi}\int_1^{cK^2\elg}\frac{1}{t}dt + \frac{1}{2\pi}\int_1^{cK^2\elg}\frac{e^{-t}}{t}dt + O(1)\\
		&= -\frac{1}{2\pi}\log{K^2} + O\left(\log\log{{n}}\right).
\end{split}		
\end{equation}
By \eqref{F1bounds}, \eqref{eqn}, \eqref{int}, and \eqref{int2} it then follows that
\begin{equation} \label{F11stinterval}
\frac{1}{\pi}\int_{1-\elg}^{1-\lgn}F_1 = \frac{1}{2\pi}\log{\left(\frac{n}{K^2}\right)} + O\left(\log\log{{n}}\right).
\end{equation}		
To handle the interval from $(-1+\lgn, -1+\elg)$ we will substitute in $-x=-1+y$, where $x\in (1-\elg,1-\lgn)$.  Then
\begin{flalign*}
A &=       \int_{-\pi}^{\pi}\frac{1-(-x)^{n+1}e^{-i(n+1)\phi}}{1+xe^{-i\phi}} \cdot \frac{1-(-x)^{n+1}e^{i(n+1)\phi}}{1+xe^{i\phi}}f(\phi)d\phi,\\
B &= \int_{-\pi}^{\pi}\left(\frac{1-(-x)^{n+1}e^{-i(n+1)\phi}}{1+xe^{-i\phi}}\right)\\
  &  \quad \cdot \left(\frac{-(n+1)(-x)^ne^{i(n+1)\phi}(1+xe^{i\phi})-(1-(-x)^{n+1}e^{i(n+1)\phi})(-e^{i\phi})}{(1+xe^{i\phi})^2}\right)f(\phi)d\phi,
\end{flalign*}
and
\begin{flalign*}
C &= \int_{-\pi}^{\pi}\left(\frac{-(n+1)(-x)^ne^{-i(n+1)\phi}(1+xe^{-i\phi})-(1-(-x)^{n+1}e^{-i(n+1)\phi})(-e^{-i\phi})}{(1+xe^{-i\phi})^2}\right)\\
  &  \quad \cdot \left(\frac{-(n+1)(-x)^ne^{i(n+1)\phi}(1+xe^{i\phi})-(1-(-x)^{n+1}e^{i(n+1)\phi})(-e^{i\phi})}{(1+xe^{i\phi})^2}\right)f(\phi)d\phi.\\
\end{flalign*}  
From (3.15), (3.16), and (3.19) in \cite{Matayoshi}, we have
\begin{equation} \label{ABC-cont}
\begin{split}
A &\sim \frac{2f(\pi)}{y}\arctan{\left(\frac{g(y)}{y}\right)},\\
B &\sim -\frac{f(\pi)}{y^2}\arctan{\left(\frac{g(y)}{y}\right)},\\
C &\sim \frac{f(\pi)}{y^3}\arctan{\left(\frac{g(y)}{y}\right)},
\end{split}
\end{equation}
for $f(\phi)\in \cont$, and
\begin{equation} \label{ABC-}
\begin{split}
A &= \frac{2f(\pi)}{y}\arctan{\left(\frac{g(y)}{y}\right)} + O\left(\frac{1}{g(y)}\right),\\
B &= -\frac{f(\pi)}{y^2}\arctan{\left(\frac{g(y)}{y}\right)} + O\left(\frac{1}{yg(y)}\right),\\
C &= \frac{f(\pi)}{y^3}\arctan{\left(\frac{g(y)}{y}\right)} + O\left(\frac{1}{y^2g(y)}\right),
\end{split}
\end{equation}
for $f(\phi) \in \diff$.  We now have the expressions
\begin{equation} \label{-ac-b2cont}
\begin{split}
AC - B^2 &\sim \frac{f^2(\pi)}{y^4}\left[\arctan{\left(\frac{g(y)}{y}\right)}\right]^2,\\
\frac{\sqrt{AC - B^2}}{A} &\sim \frac{1}{2y},
\end{split}
\end{equation}
for $f(\phi)\in \cont$, and
\begin{equation} \label{-ac-b2diff}
\begin{split}
AC - B^2 &= \frac{f^2(\pi)}{y^4}\left[\arctan{\left(\frac{g(y)}{y}\right)}\right]^2 + O\left(\frac{1}{y^3g(y)}\right),\\
\frac{\sqrt{AC - B^2}}{A} &= \frac{1}{2y} + O\left(\frac{1}{g(y)}\right),
\end{split}
\end{equation}
for $f(\phi) \in \diff$.

We will again start with the simpler case when $f(\phi)\in\cont$ and $K$ is bounded.  By \eqref{ABC-cont} and \eqref{-ac-b2cont},
\[
\frac{CK^2}{2(AC - B^2)} \sim \frac{K^2y}{2f(\pi)\arctan{\left(\frac{g(y)}{y}\right)}},
\]
from which it then follows that
\begin{equation} \label{F1contneg}
\begin{split}
& \frac{1}{\pi}\int_{\lgn}^{\elg}\frac{1}{2y}\exp{\left(\frac{-K^2y}{2f(\pi)\arctan{\left(\frac{g(y)}{y}\right)}}\right)}dy\\
&\sim \frac{1}{\pi}\int_{\lgn}^{\elg}\frac{1}{2y}\left(1 - \frac{K^2y}{2f(\pi)\arctan{\left(\frac{g(y)}{y}\right)}}\right)dy\\
&\sim \frac{1}{2\pi}\log{n}.
\end{split}
\end{equation}

Next, we will assume $f(\phi)\in\diff$ and $K=o\left(\sqrt{\frac{n}{\log{\log{n}}}}\right)$.  Using \eqref{ABC-} and \eqref{-ac-b2diff} gives us
\begin{flalign*}
\frac{CK^2}{2(AC - B^2)} &= \frac{K^2}{2}\left[\frac{f(\pi)}{y^3}\arctan{\left(\frac{g(y)}{y}\right)} + O\left(\frac{1}{y^2g(y)}\right)\right]\\
	&  \quad \cdot \left[\frac{f^2(\pi)}{y^4}\left[\arctan{\left(\frac{g(y)}{y}\right)}\right]^2 + O\left(\frac{1}{y^3g(y)}\right)\right]^{-1}\\
	&= \frac{K^2y}{2f(\pi)\arctan{\left(\frac{g(y)}{y}\right)}} + O\left(\frac{K^2y^2}{g(y)}\right).
\end{flalign*}
As before, we can choose positive constants $a_1$ and $a_2$ such that
\begin{flalign*}
\frac{a_1K^2y}{2f(\pi)\arctan{\left(\frac{g(y)}{y}\right)}} &\leq \frac{K^2y}{2f(\pi)\arctan{\left(\frac{g(y)}{y}\right)}} + O\left(\frac{K^2y^2}{g(y)}\right)\\
	& \leq \frac{a_2K^2y}{2f(\pi)\arctan{\left(\frac{g(y)}{y}\right)}},
\end{flalign*}
which then yields
\begin{multline} \label{F1bounds-}
\left[\frac{1}{2y} + O\left( \frac{1}{g(y)}\right)\right]\exp{\left(\frac{-a_2K^2y}{2f(\pi)\arctan{\left(\frac{g(y)}{y}\right)}}\right)}\\
\leq F_1 \leq \left[\frac{1}{2y} + O\left( \frac{1}{g(y)}\right)\right]\exp{\left(\frac{-a_1K^2y}{2f(\pi)\arctan{\left(\frac{g(y)}{y}\right)}}\right)}.\\
\end{multline}
Now, for $i=1,2$ we have
\begin{multline} \label{eqn2}
\left[\frac{1}{2y} + O\left( \frac{1}{g(y)}\right)\right]\exp{\left(\frac{-a_iK^2y}{2f(\pi)\arctan{\left(\frac{g(y)}{y}\right)}}\right)} \\ =\frac{1}{2y}\exp{\left(\frac{-a_iK^2y}{2f(\pi)\arctan{\left(\frac{g(y)}{y}\right)}}\right)} + O\left( \frac{1}{g(y)}\right).
\end{multline}
Thus,
\begin{equation} \label{int-}
\begin{split}
& \frac{1}{\pi}\int_{\lgn}^{\elg}\Bigg[\frac{1}{2y}\exp{\left(\frac{-a_iK^2y}{2f(\pi)\arctan{\left(\frac{g(y)}{y}\right)}}\right)} + O\left(\frac{1}{g(y)}\right)\Bigg]dy\\
 			&= \frac{1}{\pi}\int_{\lgn}^{\elg}\frac{1}{2y}\exp{\left(-cK^2y\right)}dy + O\left(\log{\log{n}}\right)\\
 			&  \quad \left(\mbox{where $c=a_i\left[2f(\pi)\arctan{\left(\frac{g(y)}{y}\right)}\right]^{-1}$}\right)\\
 			&= \frac{1}{2\pi}\log{\left(\frac{n}{K^2}\right)} + O\left(\log\log{{n}}\right),
\end{split}
\end{equation}
where the last line comes from \eqref{int} and \eqref{int2}. 			
It follows from \eqref{F1bounds-}, \eqref{eqn2}, and \eqref{int-} that
\begin{equation} \label{F12ndinterval}
\frac{1}{\pi}\int_{-1+\lgn}^{-1+\elg}F_1 = \frac{1}{2\pi}\log{\left(\frac{n}{K^2}\right)} + O\left(\log\log{{n}}\right).
\end{equation}
Combined with \eqref{F1cont}, \eqref{F11stinterval}, and \eqref{F1contneg}, this completes the proof.
\end{proof}

\section{Expected Number of Level Crossings on $(-\infty,-1)$ and $(1,\infty)$}

Now that we have derived the expected number of zeros for $(-1,1)$, this last section will consider the remaining intervals $(-\infty,-1)$ and $(1,\infty)$.  We will start with the latter.  As done by Farahmand in \cite{Farahmand86} and \cite{Farahmand862}, let $x=\frac{1}{z}$.  Then, for $z\in(0,1)$ we have
\begin{equation} \label{Ainf}
\begin{split}
A\left(\frac{1}{z}\right) &= \sum_{k=0}^{n}{\sum_{j=0}^{n}{\Gamma(k-j)z^{-(k+j)}}}\\
	&= \int_{-\pi}^{\pi} \frac{1-z^{-(n+1)}e^{-i(n+1)\phi}}{1-z^{-1}e^{-i\phi}} \cdot \frac{1-z^{-(n+1)}e^{i(n+1)\phi}}{1-z^{-1}e^{i\phi}} f(\phi)d\phi\\
	&= z^{-2n}\int_{-\pi}^{\pi} \frac{1-z^{n+1}e^{i(n+1)\phi}}{1-ze^{i\phi}} \cdot \frac{1-z^{n+1}e^{-i(n+1)\phi}}{1-ze^{-i\phi}} f(\phi)d\phi,\\
\end{split}
\end{equation}
\begin{equation} \label{Binf}
\begin{split}
B\left(\frac{1}{z}\right) &= \sum_{k=0}^{n}{\sum_{j=0}^{n}{\Gamma(k-j)kz^{-(k+j-1)}}}\\
	&= \int_{-\pi}^{\pi} \frac{1-z^{-(n+1)}e^{-i(n+1)\phi}}{1-z^{-1}e^{-i\phi}}\\
	&  \quad \cdot \frac{-(n+1)z^{-n}e^{i(n+1)\phi}\left(1-z^{-1}e^{i\phi}\right) + \left(1-z^{-(n+1)}e^{i(n+1)\phi}\right)e^{i\phi}} {\left(1-z^{-1}e^{i\phi}\right)^2} f(\phi)d\phi\\
	&= -z^{-2n+1}\int_{-\pi}^{\pi} \frac{1-z^{n+1}e^{i(n+1)\phi}}{1-ze^{i\phi}}\\
	&  \quad \cdot \frac{-(n+1)\left(1-ze^{-i\phi}\right) + 1-z^{n+1}e^{-i(n+1)\phi}} {\left(1-ze^{-i\phi}\right)^2}f(\phi)d\phi,\\
\end{split}
\end{equation}
and
\begin{equation} \label{Cinf}
\begin{split}
C\left(\frac{1}{z}\right) &= \sum_{k=0}^{n}{\sum_{j=0}^{n}{\Gamma(k-j)kjz^{-(k+j-2)}}}\\
	&= \int_{-\pi}^{\pi}\frac{-(n+1)z^{-n}e^{-i(n+1)\phi}\left(1-z^{-1}e^{-i\phi}\right) + \left(1-z^{-(n+1)}e^{-i(n+1)\phi}\right)e^{-i\phi}} {\left(1-z^{-1}e^{-i\phi}\right)^2}\\
	\end{split}
\end{equation}
\begin{equation*}
\begin{split}
	&  \quad \cdot \frac{-(n+1)z^{-n}e^{i(n+1)\phi}\left(1-z^{-1}e^{i\phi}\right) + \left(1-z^{-(n+1)}e^{i(n+1)\phi}\right)e^{i\phi}} {\left(1-z^{-1}e^{i\phi}\right)^2}f(\phi)d\phi\\
	&= z^{-2n+2}\int_{-\pi}^{\pi}\frac{-(n+1)\left(1-ze^{i\phi}\right)+1-z^{n+1}e^{i(n+1)\phi}}{\left(1-ze^{i\phi}\right)^2}\\
	&  \quad \cdot \frac{-(n+1)\left(1-ze^{-i\phi}\right) +1-z^{n+1}e^{-i(n+1)\phi}}{\left(1-ze^{-i\phi}\right)^2}f(\phi)d\phi.\\
\end{split}	
\end{equation*}

As before, the first step is to get a bound for the integral of $F_2$.
\begin{lemma}\label{F2infbound}
\[
\int_1^{\infty}F_2dx = \int_{-\infty}^{-1}F_2dx = o(1).
\]
\end{lemma}
\begin{proof}
We have
\begin{equation} \label{F2intbound}
\begin{split}
\int_1^{\infty}F_2dx &\leq \frac{\sqrt{2}}{\pi}\int_1^{\infty}\frac{\left|B(x)K\right|}{A^{3/2}(x)}dx\\
			&= \frac{\sqrt{2}}{\pi}\int_0^1\frac{1}{z^2}\frac{\left|B\left(\frac{1}{z}\right)K\right|}{A^{3/2}\left(\frac{1}{z}\right)}dz.\\
\end{split}
\end{equation}
Let $c_1$ and $c_2$ be as in the proof of Lemma \ref{F2bound}.  Then, for $z\in(-1,0)\cup(0,1)$,
\begin{flalign*}
\left|B\left(\frac{1}{z}\right)\right| &\leq n|z|^{-2n+1}\sum_{k=0}^n\sum_{j=0}^n\Gamma(k - j)|z|^{2n-k-j}\\
		&=    n|z|^{-2n+1}A(|z|)\\
		&\leq  c_1n|z|^{-2n+1}\frac{1-z^{2n+2}}{1-z^2},
\end{flalign*}
where the last line is given by \eqref{upA}.  Also,
\begin{flalign*}
A\left(\frac{1}{z}\right) 
	&\geq z^{-2n}\frac{c_2}{2\pi}\int_{-\pi}^{\pi}\frac{1-z^{n+1}e^{i(n+1)\phi}}{\left(1-ze^{i\phi}\right)}\cdot\frac{1-z^{n+1}e^{-i(n+1)\phi}}{\left(1 - ze^{-i\phi}\right)}d\phi\\
	&= c_2z^{-2n}\frac{1-z^{2n+2}}{1-z^2}.
\end{flalign*}
Thus,
\[
\frac{\left|B\left(\frac{1}{z}\right)\right|}{A^{3/2}\left(\frac{1}{z}\right)} \leq cn|z|^{n+1}\sqrt{\frac{1-z^2}{1-z^{2n+2}}}.
\]
Consider the interval $(0,1-\frac{1}{\sqrt{n}})$.  Recalling that $K=o\left(\sqrt{\frac{n}{\log{\log{n}}}}\right)$, the above inequality yields
\begin{flalign*}
&\frac{\sqrt{2}}{\pi}\int_0^{1-\frac{1}{\sqrt{n}}}\frac{1}{z^2}\frac{\left|B\left(\frac{1}{z}\right)K\right|}{A^{3/2}\left(\frac{1}{z}\right)}dz\\
&\leq c|K|\int_0^{1-\frac{1}{\sqrt{n}}}nz^{n-1}\sqrt{\frac{1-z^2}{1-z^{2n+2}}}\\
&\leq c|K|\left(1 - \frac{1}{\sqrt{n}}\right)^n\\
&=    o(1).
\end{flalign*}
Next, for $z\in(1-\frac{1}{\sqrt{n}},1)$ we have
\begin{flalign*}
&\frac{\sqrt{2}}{\pi}\int_{1-\frac{1}{\sqrt{n}}}^1\frac{1}{z^2}\frac{\left|B\left(\frac{1}{z}\right)K\right|}{A^{3/2}\left(\frac{1}{z}\right)}dz\\
&\leq  c|K|\int_{1-\frac{1}{\sqrt{n}}}^1nz^{n-1}\sqrt{\frac{1-z^2}{1-z^{2n+2}}}\\
&=     \left.c|K|z^n\sqrt{\frac{1-z^2}{1-z^{2n+2}}}\right|_{1-\frac{1}{\sqrt{n}}}^1 - c|K|\int_{1-\frac{1}{\sqrt{n}}}^1z^n \frac{d}{dz}\left(\sqrt{\frac{1-z^2}{1-z^{2n+2}}}\right)dz\\
&= o(1),
\end{flalign*}			
where the last line follows from the fact that
\[
\frac{d}{dz}\left(\sqrt{\frac{1-z^2}{1-z^{2n+2}}}\right) = O\left(\sqrt{n}\right)
\]
on $z\in(1-\frac{1}{\sqrt{n}},1)$.  Applying \eqref{F2intbound}, this proves the result for $(1,\infty)$.  Noting that the same argument works for $-z$, the result then follows for $(-\infty,-1)$ as well.
\end{proof}

The next lemma will evaluate the integral of $F_1$.
\begin{lemma} \label{infcross}
\renewcommand{\theenumi}{\roman{enumi}}
\renewcommand{\labelenumi}{(\theenumi)}
\begin{enumerate}
\item For $f\in\cont$,
\[
\int_1^{\infty}F_1dx = \int_{-\infty}^{-1}F_1dx \sim \frac{1}{2\pi}\log{n}.
\]
\item For $f\in\diff$,
\[
\int_1^{\infty}F_1dx = \int_{-\infty}^{-1}F_1dx = \frac{1}{2\pi}\log{n} + O\left(\log{\log{n}}\right).
\]
\end{enumerate}
\end{lemma}
\begin{proof}
We will prove the result assuming that $f\in\diff$; the resulting argument will require only a few minor changes to prove the claim for $f\in\cont$.  As in Lemma \ref{-1to1}, this will be done by bounding the true asymptotic value between an upper and a lower bound.  To start, we have the inequality 
\begin{equation*}
 \int_{1}^{\infty} F_1dx \leq \frac{1}{\pi}\int_{1}^{\infty}\frac{\sqrt{A(x)C(x) - B^2(x)}}{A(x)}dx.
\end{equation*}
Notice that the expression on the right is simply the expected number of real zeros of $P_n(x)$ on $(1,\infty)$.  Similarly,
\begin{equation*}
 \int_{-\infty}^{-1} F_1dx \leq \frac{1}{\pi}\int_{-\infty}^{-1}\frac{\sqrt{A(x)C(x) - B^2(x)}}{A(x)}dx,
\end{equation*}
where now the expression on the right is the expected number of real zeros of $P_n(x)$ on $(-\infty,-1)$.  Thus, Theorem 1.1 in \cite{Matayoshi} yields the upper bounds
\begin{equation} \label{F1infupper}
\begin{split}
\int_1^{\infty}F_1dx &\leq \frac{1}{2\pi}\log{n} + O\left(\log{\log{n}}\right),\\
\int_{-\infty}^{-1}F_1dx &\leq \frac{1}{2\pi}\log{n} + O\left(\log{\log{n}}\right).
\end{split}
\end{equation}
The rest of the proof will be devoted to the derivation of a lower bound.

Consider the interval $(1 - \elg, 1 - \lgn)$.  Let $z=1-y$, and recall that $g(y)=y\frac{\log{n}}{\log{\log{n}}}$.  We will next need to make use of the asymptotic formulas
\begin{equation} \label{asymp}
\begin{split}
\int_{-\pi}^{\pi}\frac{f(\phi)d\phi}{\left(1-ze^{i\phi}\right)\left(1-ze^{-i\phi}\right)} &= \frac{2f(0)}{y}\arctan{\left(\frac{g(y)}{y}\right)} + O\left(\frac{1}{g(y)}\right),\\
\int_{-\pi}^{\pi}\frac{f(\phi)d\phi}{\left(1-ze^{i\phi}\right)\left(1-ze^{-i\phi}\right)^2} &= \frac{f(0)}{y^2}\arctan{\left(\frac{g(y)}{y}\right)} + O\left(\frac{1}{yg(y)}\right),\\
\int_{-\pi}^{\pi}\frac{f(\phi)d\phi}{\left(1-ze^{i\phi}\right)^2\left(1-ze^{-i\phi}\right)^2} &= \frac{f(0)}{y^3}\arctan{\left(\frac{g(y)}{y}\right)} + O\left(\frac{1}{y^2g(y)}\right),
\end{split}
\end{equation}
which are derived in the proof of Lemma 3.1 in \cite{Matayoshi}.  Combining \eqref{asymp} with \eqref{Ainf}, \eqref{Binf}, and \eqref{Cinf}, and after some tedious algebra, we can obtain the expression
\begin{equation*}
\begin{split}
 & A\left(\frac{1}{z}\right)C\left(\frac{1}{z}\right)-B^2\left(\frac{1}{z}\right)\\
 &= z^{-4n+2}\Bigg[\int_{-\pi}^{\pi}\frac{f(\phi)d\phi}{\left(1-ze^{i\phi}\right)\left(1-ze^{-i\phi}\right)} \cdot \int_{-\pi}^{\pi} 											\frac{f(\phi)d\phi}{\left(1-ze^{i\phi}\right)^2\left(1-ze^{-i\phi}\right)^2}\\
 & 	\quad			    - \left(\int_{-\pi}^{\pi}\frac{f(\phi)d\phi}{\left(1-ze^{i\phi}\right)\left(1-ze^{-i\phi}\right)^2}\right)^2\\
 &  \quad + O\left((n+1)z^{n+1} \int_{-\pi}^{\pi}\frac{f(\phi)d\phi}{\left(1-ze^{i\phi}\right)\left(1-ze^{-i\phi}\right)} \cdot  													\int_{-\pi}^{\pi}\frac{f(\phi)d\phi}{\left(1-ze^{i\phi}\right)\left(1-ze^{-i\phi}\right)^2}\right) \Bigg]\\
 &= (1-y)^{-4n+2}\left[\frac{f^2(0)}{y^4}\arctan^2{\left(\frac{g(y)}{y}\right)} + O\left(\frac{1}{y^3g(y)}\right)\right].			 
\end{split}
\end{equation*}
Thus,
\begin{equation} \label{abcinf}
\frac{\sqrt{A\left(\frac{1}{z}\right)C\left(\frac{1}{z}\right)-B^2\left(\frac{1}{z}\right)}}{A\left(\frac{1}{z}\right)} = (1-y)\left[\frac{1}{2y} + O\left(\frac{1}{g(y)}\right)\right].
\end{equation}
Also, if we refer to \eqref{asymp} once more,
\begin{equation} \label{cinf}
C\left(\frac{1}{z}\right) \sim (1-y)^{-2n+2}\frac{2(n+1)^2f(0)}{y}\arctan{\left(\frac{g(y)}{y}\right)}.
\end{equation}
Applying \eqref{fullKac} we then have
\begin{flalign*}
&\int_{1}^{\infty} F_1dx=\\
&= \frac{1}{\pi}\int_{0}^{1}\frac{1}{z^2}\frac{\sqrt{A\left(\frac{1}{z}\right)C\left(\frac{1}{z}\right)-B^2\left(\frac{1}{z}\right)}}{A\left(\frac{1}{z}\right)}\exp\left(-\frac{K^2C\left(\frac{1}{z}\right)}{2\left(A\left(\frac{1}{z}\right)C\left(\frac{1}{z}\right) - B^2\left(\frac{1}{z}\right)\right)}\right)dz\\
&\geq \frac{1}{\pi}\int_{1-\elg}^{1-\lgn}\frac{1}{z^2}\frac{\sqrt{A(\frac{1}{z})C(\frac{1}{z}) - B^2(\frac{1}{z})}} {A(\frac{1}{z})}\exp\left(-\frac{K^2C(\frac{1}{z})}{2\left(A(\frac{1}{z})C(\frac{1}{z}) - B^2(\frac{1}{z})\right)}\right)dz\\
			 &= \frac{1}{\pi}\int_{\lgn}^{\elg}\Bigg[\frac{1}{2y(1-y)}\left[1 + O\left(K^2(n+1)^2(1-y)^{2n}y^3\right)\right] + O\left(\frac{1}{g(y)}\right)\Bigg]dy\\
			 &= \frac{1}{2\pi}\log{n} + O\left(\log{\log{n}}\right).
\end{flalign*}
Noting that almost the exact same argument holds for $-z$,
\begin{equation*}
\int_{-\infty}^{-1} F_1dx \geq \frac{1}{2\pi}\log{n} + O\left(\log{\log{n}}\right),
\end{equation*}
as well.  Combined with \eqref{F1infupper}, the claim then follows.
\end{proof}

\begin{proof}[Proof of Theorem \ref{Kcrossings}]
Combining the results of Lemmas \ref{F2bound}, \ref{Kbound}, \ref{-1to1}, \ref{F2infbound}, and \ref{infcross}, Theorem \ref{Kcrossings} now follows.
\end{proof}

\section{Acknowledgments}
The author would like to thank his thesis advisor, Professor Michael Cranston, for his guidance and support with the subject.  The author is also grateful to Professor Stanislav Molchanov for suggesting the idea of using the spectral density of the covariance function.  Finally, a thank you is owed to Mr. Phillip McRae for reading through a copy of this manuscript.

\end{document}